\documentclass[11pt,twoside]{amsart}
\usepackage{amssymb}

\usepackage[latin1]{inputenc}
\usepackage[T1]{fontenc}
\usepackage{mathtools}
\usepackage{graphicx}
\usepackage{tikz}
\usepackage{mathabx}
\usetikzlibrary{chains}
\usepackage[OT2,T1]{fontenc}

\PassOptionsToPackage{pdfusetitle,pagebackref,colorlinks}{hyperref}
\usepackage{bookmark}
\hypersetup{
  linkcolor={red!70!black},
  citecolor={green!70!black},
  urlcolor={blue!80!black}
}

\usepackage[latin1]{inputenc}
\usepackage{amsmath}
\usepackage{amsthm}

\usepackage[all]{xy}
\usepackage{hyperref}
\newtheorem{thm}{Theorem}[section]

\newtheorem{prop}[thm]{Proposition}

\newtheorem{lem}[thm]{Lemma}
\newtheorem{cor}[thm]{Corollary}
\newtheorem{conj}[thm]{Conjecture}
\numberwithin{equation}{section}

\theoremstyle{definition}

\newtheorem{remark}[thm]{Remark}

\DeclareFontFamily{U}{mathc}{}
\DeclareFontShape{U}{mathc}{m}{it}%
{<->s*[1.03] mathc10}{}
\DeclareMathAlphabet{\mathcal}{U}{mathc}{m}{it}

\newcommand{\Db}{{\rm D}^{\rm b}}

\newcommand{\Br}{{\rm Br}}

\newcommand{\Pic}{{\rm Pic}}

\newcommand{\sign}{\sigma}

\newcommand{\coker}{{\rm coker}}

\newcommand{\cal}{\mathcal}
\newcommand{\ka}{{\cal A}}

\newcommand{\kc}{{\cal C}}

\newcommand{\kf}{{\cal F}}

\newcommand{\ko}{{\cal O}}
\newcommand{\kp}{{\cal P}}
\newcommand{\kq}{{\cal Q}}

\newcommand{\kz}{{\cal Z}}
\newcommand{\GG}{\mathbb{G}}

\newcommand{\ZZ}{\mathbb{Z}}
\newcommand{\QQ}{\mathbb{Q}}

\newcommand{\PP}{\mathbb{P}}

\newcommand{\Gm}{\mathbb{G}_m}

\DeclareSymbolFont{cyrletters}{OT2}{wncyr}{m}{n}
\DeclareMathSymbol{\Sha}{\mathalpha}{cyrletters}{"58}

\renewcommand{\to}{\xymatrix@1@=15pt{\ar[r]&}}
\newcommand{\lto}{\xymatrix@1@=15pt{&\ar[l]}}
\renewcommand{\rightarrow}{\xymatrix@1@=15pt{\ar[r]&}}
\renewcommand{\mapsto}{\xymatrix@1@=15pt{\ar@{|->}[r]&}}
\newcommand{\mapslto}{\xymatrix@1@=15pt{&\ar@{|->}[l]&}}
\renewcommand{\twoheadrightarrow}{\xymatrix@1@=18pt{\ar@{->>}[r]&}}

\renewcommand{\hookrightarrow}{\xymatrix@1@=15pt{\ar@{^(->}[r]&}}
\newcommand{\hook}{\xymatrix@1@=15pt{\ar@{^(->}[r]&}}

\newcommand{\congpf}{\xymatrix@1@=15pt{\ar[r]^-\sim&}}
\renewcommand{\cong}{\simeq}

\newcommand{\TBC}[1]{}
\newcommand{\EV}[1]{}

\makeatletter
\def\blfootnote{\xdef\@thefnmark{}\@footnotetext}
\makeatother

\begin{document}

\title[]{Derived categories of Fano varieties of lines}

\author[A.\ Bottini \& D.\ Huybrechts]{Alessio Bottini \& Daniel Huybrechts}

\address{Max-Planck Institute for Mathematics, Vivatsgasse 7, 53111 Bonn, Germany \& Mathematical Institute and Hausdorff Center for Mathematics,
University of Bonn, Endenicher Allee 60, 53115 Bonn, Germany}
\email{bottini@mpim-bonn.mpg.de \& huybrech@math.uni-bonn.de}

\begin{abstract}  \vspace{-2mm} We gather evidence for a conjecture of Galkin predicting the derived category of the Fano variety $F_X$ of lines contained in a smooth cubic fourfold $X\subset\PP^5$ to be equivalent to the Hilbert square $\ka_X^{[2]}$ of the Kuznetsov component $\ka_X\subset\Db(X)$. We prove the conjecture for generic Fano varieties admitting a rational Lagrangian fibration and show that the natural Hodge structures of weight two associated with $F_X$ and $\ka_X^{[2]}$ are isometric.\end{abstract}
%\dedicatory{ ERC Synergy Grant HyperK (ID 854361).}

\maketitle
\blfootnote{Both authors have been supported by the ERC Synergy Grant HyperK (ID 854361).}

With any smooth cubic hypersurface $X\subset\PP^5$ of dimension four there is associated
the Fano variety of lines $F_X$ which is a projective hyperk\"ahler manifold  of dimension four deformation equivalent
to the Hilbert square $S^{[2]}$ of a K3 surface $S$. While derived categories of K3 surfaces
have been studied intensively, notably by Mukai \cite{Mukai}, Orlov \cite{Orlov}, and Bridgeland \cite{BK3},  not much is known about  derived categories of higher-dimensional hyperk\"ahler manifolds. The derived category $\Db(S^{[2]})$ of the Hilbert square of a K3 surface is a good place to start and another interesting and accessible case is the derived category $\Db(F_X)$ of the Fano variety of lines.\smallskip

Conjecturally, these two kinds of derived categories are related by means of the Hilbert square of the Kuznetsov component $\ka_X$. By definition, the K3 category $\ka_X\subset\Db(X)$ is the admissible full triangulated sub-category of all complexes right orthogonal to the three line bundles $\ko_X,\ko_X(1),\ko_X(2)\in\Db(X)$, see \cite{Kuz} or \cite[Ch.\ 7.1]{HuyComp}. By work of Addington and Thomas \cite{AddTh}, we know that for cubics in a countable union of hypersurfaces $\kc_d\subset\kc$ in the moduli space $\kc$ of all smooth cubic fourfolds the Kuznetsov component $\ka_X$ is in fact
equivalent to the derived category $\Db(S)$ of a K3 surface $S$ or, more generally, to the
derived category $\Db(S,\alpha)$ of a K3 surface $S$ twisted by a Brauer class $\alpha\in \Br(S)$, cf.\ \cite{HuyComp}. Imitating the passage from a K3 surface $S$ to its Hilbert square
$S^{[2]}$, one associates with %the Kuznetsov component
$\ka_X\subset \Db(X)$
 its Hilbert square $\ka_X^{[2]}\subset\Db(X^{[2]})$, cf.\  \S\! \ref{sec:Hilbsquare}. Note that
 by \cite{BKR}, 
$\ka_X^{[2]}\cong\Db(S^{[2]})$ if $\ka_X\cong \Db(S)$.\smallskip

The following  conjecture has been formulated independently by various people but probably first in talks by Galkin almost ten years ago. 
%He also suggested a line of attack which, however, has not yet led to a proof.

\begin{conj}\label{conj}
The Fano variety $F_X$  of lines contained in a smooth cubic fourfold $X\subset\PP^5$ is derived equivalent
to the Hilbert square of the Kuznetsov component, i.e.\ there exists an exact linear equivalence
\begin{equation}\label{eqn:conj}
\Db(F_X)\cong\ka_X^{[2]}.
\end{equation}
\end{conj}

It is not difficult to see that $\Db(F_X)$ and $\ka_X^{[2]}$ define the same class in the Grothendieck ring of dg-enhanced triangulated categories, i.e.
$$[\Db(F_X)]=[\ka_X^{[2]}] \in K_0(\text{dg-cat}).$$
This observation goes back to Galkin and Shinder \cite{GS} and was upgraded by Belmans, Fu, and Raedschelders \cite{BFR} to the statement that both categories are admissible
sub-categories of  the same triangulated category $\Db(X^{[2]})$
with their respective complements both built out of five copies of $\Db(X)$, cf.\ \cite[Ch.\ 7]{HuyCubics}.

\subsection{The conjecture for dense sets of cubics} Let us start by first noting the following easy result which is certainly
known to the experts, cf.\ \cite{Popov}. It proves Conjecture \ref{conj} for all smooth cubic fourfolds $X$
contained in the countable $\bigcup_{\rm{(\ast\ast\ast)}}\kc_d$ union of all Hassett divisors $\kc_d\subset \kc$ for which $d$ satisfies the
numerical condition $\rm{(\ast\!\ast\!\ast)}$, i.e.\ $d$ is of the form $(2n^2+2n+2)/a^2$ for certain $n$ and $a$, which geometrically means that the Fano variety $F_X$ is birational to a Hilbert scheme, cf.\ \cite[Ch.\ 6.5]{HuyCubics}.

\begin{thm}\label{thm1}
Assume the Fano variety $F_X$ of a smooth cubic fourfold $X\subset\PP^5$
is birational to the Hilbert scheme $S^{[2]}$ of a K3 surface. Then $\Db(F_X)\cong \ka_X^{[2]}$.
\end{thm}

Indeed, a birational correspondence between two four-dimensional hyperk\"ahler manifolds
is a composition of simple Mukai flops which by results of Kawamata \cite{Ka}
and Namikawa \cite{Na} are known to induce derived equivalences, cf.\ \cite[Ch.\ 11]{HuyFM}.
Hence, $\Db(F_X)\cong \Db(S^{[2]})$. More general results have recently been proved by
Maulik, Shen, Yin, and Zhang \cite{MSYZ}. On the other hand, the existence of a birational correspondence $F_X\sim S^{[2]}$ also implies that $\ka_X\cong\Db(S)$. Indeed, by results of Addington \cite{Add} and Hassett \cite{Ha}, the existence of a birational correspondence $F_X\sim S^{[2]}$ is equivalent to $X$ being contained in a Hassett divisor $\kc_d$ with $d$ satisfying  $\rm{(\ast\!\ast\!\ast)}$. However, this immediately implies that $X\in \kc_d$
with $d$ satisfying Hassett's condition $\rm{(\ast\ast)}$, i.e.\ $d/2=\prod p^{n_p}$ with $n_p=0$ for all $p\equiv 2\,(3)$ and $n_3\leq1$, which by the work of Addington and Thomas \cite{AddTh} mentioned before and further work of Bayer et al \cite{BLMNPS} implies $\ka_X\cong\Db(S)$ for some K3 surface $S$ (in fact for the same $S$ for which $F_X\sim S^{[2]}$). Combined with the BKR-equivalence $\Db(S^{[2]})\cong\Db(S)^{[2]}$, see \cite{BKR},  we obtain $$\Db(F_X)\cong\Db(S^{[2]})\cong\Db(S)^{[2]}\cong\ka_X^{[2]}.$$

The main result of this note is the following theorem which proves Conjecture \ref{conj} 
for $X$ contained in another countable union $\bigcup_{\rm(\ast\ast'')}\kc_d$ of Hassett divisors.

\begin{thm}\label{thm2}
Assume the Fano variety $F_X$ admits a rational Lagrangian fibration and is generic with this property. Then $\Db(F_X)\cong\ka_X^{[2]}$.
\end{thm}

The existence of a rational Lagrangian fibration of $F_X$ is equivalent to requiring
$X\in \kc_d$ with $d/2$ a perfect square, see \cite[Cor.\ 6.5.24]{HuyCubics}, i.e.\
$d/2=\prod p^{n_p}$ with  $n_p\equiv0\, (2)$ for all $p$. This numerical condition is called $\rm(\ast\ast'')$. According to \cite{HuyComp}, the weaker condition $\rm (\ast\ast')$ that $n_p\equiv 0\,(2)$ only for all $p\equiv2\, (3)$ is known to be equivalent to $\ka_X\cong\Db(S,\alpha)$ for some twisted K3 surface $(S,\alpha)$. Again, this was originally proved only for generic $X$ but now holds for all $X$ by \cite{BLMNPS}.\smallskip

Note the following chain of implications between these numerical conditions
$$\xymatrix@R=12pt@C=15pt{\stackrel{{\color{red}\checkmark}}{{\rm(\ast\!\ast\!\ast)}}\ar@{=>}[r]&{\rm (\ast\ast)}\ar@{=>}[r]&{\rm (\ast\ast')}&\stackrel{{\color{red}\checkmark}}{{\rm (\ast\ast'')}}\ar@{=>}[l].}$$
We marked those for which, by Theorems \ref{thm1} and \ref{thm2}, Conjecture \ref{conj} has now been verified.
Naturally, one would expect the case ${\rm (\ast\ast)}$, i.e.\ when $\ka_X\cong\Db(S)$,
to be the  most accessible one among all cubics for Conjecture \ref{conj}. However, even this case remains open in full generality and the smallest such $d$ for which the conjecture is not settled by the results above is $d=74$. It would be even more interesting to settle the conjecture when
$\ka_X$ is of the form  $\Db(S,\alpha)$, but not $\Db(S)$, so  ${\rm (\ast\ast')}$ holds but not ${\rm (\ast\ast)}$.  Here, the smallest such $d$ for which also ${\rm (\ast\ast'')}$ does not hold is $d=24$, in which case $\ka_X\cong\Db(S,\alpha)$ with $S$ a K3 surface of degree six and $\alpha$ of order two. None of the available techniques seem to say anything about 
cases beyond ${\rm (\ast\ast')}$, so without any K3 surface around, with the smallest possible value being $d=12$.\smallskip

To the best of our knowledge, not a single smooth cubic fourfold $X$ is known for which 
(\ref{eqn:conj}) has been established and which is not covered by one of the above  two theorems.
%%%%%%%%%%%%%%%%%%%%%
\subsection{Hodge theoretic evidence} Already for K3 surfaces, Mukai \cite{Mukai} and Orlov \cite{Orlov} observed that 
$\Db(S)$ is uniquely determined by the Hodge structure  $\widetilde H(S,\ZZ)$ of weight two
which is essentially $H^\ast(S,\ZZ)$ with the hyperbolic plane $(H^0\oplus H^4)(S,\ZZ)\cong U$ declared to be of type $(1,1)$. For the K3 category $\ka_X$ one defines a similar
Hodge structure $\widetilde H(\ka_X,\ZZ)$ of weight two as the orthogonal
complement of the three line bundles $\ko_X(i)$, $i=0,1,2$, in $K_{\text{top}}(X)$ with the $(2,0)$-part given by the pre-image of $H^{3,1}(X)$ under the Mukai vector
$v\colon K_{\text{top}}(X)\to H^\ast(X,\QQ)$, cf.\ \cite[Ch.\ 5.4]{HuyCubics}.\smallskip

We define similarly the \emph{naive Hodge structure} $\widetilde H(Z,\ZZ)$ of weight two of a hyperk\"ahler manifold $Z$ of ${\rm K3}^{[2]}$-type  as $$\widetilde H(Z,\ZZ)\coloneqq H^2(Z,\ZZ)\oplus U$$ with $U\subset \widetilde H^{1,1}(Z,\ZZ)$. However, as discussed first by Taelman \cite{Taelman} and later by Beckmann \cite{Beck}, the structure has to be twisted slightly in order to render it useful in the study of derived categories. Luckily,
as we will argue, this
additional twist can be ignored for $F_X$, cf.\ \S\! \ref{sec:BTtwist}.\smallskip

For a K3 surface $S$ itself, one thus has $\widetilde H(S^{[2]},\ZZ)=\widetilde H(S,\ZZ)\oplus \ZZ\cdot\delta$ with $2\delta$ being the class of the exceptional divisor of the Hilbert--Chow morphism $S^{[2]}\to S^{(2)}$. Imitating this description, we define the naive weight two Hodge structure of the Hilbert square of $\ka_X$ as $$\widetilde H(\ka_X^{[2]},\ZZ)\coloneqq \widetilde H(\ka_X,\ZZ)\oplus^\perp \ZZ\cdot \delta,$$
where $\delta$ is by definition an integral class of type $(1,1)$ with $(\delta)^2=-2$.
Clearly, if $\ka_X\cong\ka_{X'}$, then also $\ka_X^{[2]}\cong\ka_{X'}^{[2]}$, which is
reflected by a Hodge isometry $\widetilde H(\ka_X^{[2]},\ZZ)\cong\widetilde H(\ka_{X'}^{[2]},\ZZ)$. It is expected that any Fourier--Mukai equivalence
$\ka_X^{[2]}\cong\ka_{X'}^{[2]}$ between just the Hilbert squares also induces such a Hodge isometry, cf.\ \S\! \ref{sec:MukaiLattice}.  In the same sense, also the following result is seen as evidence for (\ref{eqn:conj}).

\begin{thm}\label{thm3}
For every smooth cubic fourfold $X\subset\PP^5$ there exists
a Hodge isometry between the naturally associated naive weight two
Hodge structures:
$$\widetilde H(F_X,\ZZ)\cong \widetilde H(\ka_X^{[2]},\ZZ).$$
\end{thm}
We hope that this result will allow us eventually to detect the natural derived involution of $\Db(F_X)$, first cohomologically and then on the level of categories, that would come from an equivalence $\Db(F_X)\cong\ka_X^{[2]}$ and a result of Elagin \cite{Elagin}, see \S\! \ref{sec:sigma}.

\subsection{Conventions and background}  Throughout, $X\subset\PP^5$ 
and $S$ denote a smooth cubic hypersurface of dimension four and a projective K3 surface. Furthermore, $F_X$ is the Fano variety of lines contained in $X$, see \cite{HuyCubics} for a general reference, and $S^{[2]}$ denotes the Hilbert scheme of length two subschemes in $S$.\smallskip

 By $\Db(X)$ and $\Db(S)$ we denote the bounded derived categories of coherent sheaves on $X$ resp.\ $S$ and $\ka_X\subset\Db(X)$ is the Kuznetsov component. We will also freely work with derived categories $\Db(S,\alpha)$ of $\alpha$-twisted sheaves, \cite[Ch.\ 16.4]{HuyK3} for a survey and references, and with moduli spaces of semi-stable $\alpha$-twisted sheaves, see e.g.\ \cite{Yos}.\smallskip

The categorical arguments all work over any field, but as soon as Hodge theory is involved we have to work over the complex numbers.  In particular, 
$\widetilde H(S,\alpha,\ZZ)$ is the Hodge structure of a twisted K3 surface $(S,\alpha)$ and $\widetilde H(\ka_X,\ZZ)$ is the K3 Hodge structure associated with $\ka_X$.\smallskip

The Hassett divisor $\kc_d$ in the moduli space of smooth cubic fourfolds $\kc$ parametrises all those cubics $X$ for which the algebraic part $H^{2,2}(X,\ZZ)$
contains a certain sublattice $K_d$ of rank two and discriminant $d$, see the orginal \cite{Ha} or the survey \cite[Ch.\ 6.5]{HuyCubics}.

%\subsection{Acknowledgements} Do we wish to thank anybody in particular?

%%%%%%%%%%%%%%%%%%
\section{The Hilbert square of the Kuznetsov component}

The main goal of this article is to study the link between $\Db(F_X)$ and $\ka_X^{[2]}$
for a smooth cubic fourfold $X\subset \PP^5$. We focus on four-dimensional cubic hypersurfaces, as in any other dimension
the link between the two categories will be less direct. For example, for a smooth cubic threefold $Y\subset\PP^4$ the classes of these two categories in the Grothendieck group satisfy  $[\ka_Y^{[2]}]=[\Db(F_Y)]+[\Db(Y)]\in K_0(\text{dg-cat})$.\smallskip

In this section we shall recall the definition of the Hilbert square $\ka_X^{[2]}$ and introduce its Hodge structure. The main result is a description of $\ka_X^{[2]}$ in the case that $\ka_X$ is the derived category of a twisted K3 surface $(S,\alpha)$.

%%%%%%%%%%%%%%%%%%%%
\subsection{Defining the Hilbert square}\label{sec:Hilbsquare}
Symmetric powers of dg-categories have first been considered by Ganter and Kapranov \cite{GK} and  Elagin \cite[Thm.\ 6.10]{Elagin},
see also the more recent treatment by Beckmann and Oberdieck \cite{BeckOber}. For admissible sub-categories of derived
categories of coherent sheaves, the construction can be made more explicit and more geometric as follows.\smallskip

Starting with the semi-orthogonal decomposition, cf.\ \cite[Ch.\ 7.3]{HuyCubics},
$$\Db(X)=\langle\ka_X,\ko_X,\ko_X(1),\ko_X(2)\rangle,$$ Kuznetsov \cite{Kuz} defines the admissible sub-category
$\ka_X\boxtimes\ka_X\subset\Db(X\times X)$ as the right orthogonal complement
of all objects of the form $\kf\boxtimes\ko_X(i)$ and $\ko_X(i)\boxtimes\kf$ with $i=0,1,2$
and $\kf\in \Db(X)$ arbitrary. The natural ${\mathfrak S}_2$-action on $X\times X$
defined by $\iota(x_1,x_2)=(x_2,x_1)$
induces an action of ${\mathfrak S}_2$ on $\ka_X\boxtimes\ka_X$ and  the Hilbert (or symmetric) square is defined as the equivariant category
$$\ka_X^{[2]}\coloneqq(\ka_X\boxtimes\ka_X)_{{\mathfrak S}_2}.$$
Thus, objects in $\ka_X^{[2]}$ are objects in $\ka_X\boxtimes \ka_X\subset\Db(X\times X)$ together with an ${\mathfrak S}_2$-linearisation and morphisms are the invariant
ones in $\ka_X\boxtimes\ka_X$. The category $\ka_X^{[2]}$ inherits the structure of a linear triangulated category, cf.\  \cite{Elagin}.

\begin{remark}
The higher symmetric powers $\ka_X^{[n]}$ are defined similarly by first defining
$\ka^{\,\boxtimes\, n}\subset \Db(X^n)$ and then taking its equivariant category with respect to the natural action of ${\mathfrak S}_n$. 
Popov \cite{Popov} extends Galkin's conjecture to the prediction that the LLSvS variety $Z_X$, a hyperk\"ahler manifold deformation equivalent to the eight-dimensional Hilbert scheme of a K3 surface, cf.\ \cite{LLSvS}, is derived equivalent to $\ka_X^{[4]}$.
\end{remark}

Let us also recall that $\ka_X^{[2]}$ comes in turn with an involution $\sign\colon \ka_X^{[2]}\congpf \ka_X^{[2]}$. By definition,
it sends an invariant object $E\in \ka_X\boxtimes\ka_X$ with a linearisation $\rho\colon E\to \iota^\ast E$  to $E$ with the linearisation $-\rho$. As proved by Elagin \cite{Elagin},
the equivariant category of $\ka_X^{[2]}$ with respect to this involution gives back $\ka_X\boxtimes\ka_X$:
$$(\ka_X^{[2]})_{{\mathfrak S}_2}\cong\ka_X\boxtimes\ka_X.$$
In \S\! \ref{sec:sigma} we will discuss the action of $\sigma$ on the cohomology of $\ka_X^{[2]}$.
%%%%%%%%%%%%%%%%%%%%%

%%%%%%%%%%%%%%%%%%%%%%
\subsection{Twisted BKR}
Bridgeland, King, and Reid \cite{BKR} describe the derived category $\Db(S^{[n]})$ of the Hilbert scheme of a K3 surface as the equivariant category of $S^n$ with respect to the natural ${\mathfrak S}_n$-action. Since, $\Db(S^{n})\cong \Db(S)^{\,\boxtimes\, n}
$, we have 
\begin{equation}\label{eqn:BKR}
\Db(S^{[n]})\cong\Db(S)^{[n]}.
\end{equation}

 We will need a twisted version of this fact. More precisely, if $S$ comes with a Brauer class $\alpha\in \Br(S)$ and
$\alpha^{[n]}\in \Br(S^{[n]})$ is the induced class via the natural isomorphism
$\Br(S)\cong\Br(S^{[n]})$, then the result we need compares
$\Db(S^{[n]},\alpha^{[n]})$ with the Hilbert product $\Db(S,\alpha)^{[n]}$.

\begin{prop}\label{prop:twBKR}
Any exact linear equivalence $\ka_X\cong\Db(S,\alpha)$ 
naturally induces an exact linear equivalence
$$\ka_X^{[n]}\cong\Db(S^{[n]},\alpha^{[n]}).$$
\end{prop}

\begin{proof} By definition, $\ka_X^{[n]}$ is the equivariant category of the
product $$\ka_X^{\,\boxtimes\, n}=\Db(S,\alpha)^{\,\boxtimes\, n}\cong\Db(S^n,\alpha^{\times n})$$ with respect to the natural ${\mathfrak S}_n$-action, where we use the shorthand $\alpha^{\times n}\coloneqq \prod p_i^\ast\alpha$. We will use that there exists a Brauer class $\alpha^{(n)}\in \Br(S^{(n)})$ that pulls back to
$\alpha^{[n]}$ and $\alpha^{\times n}$ under the two projections
$$\xymatrix{S^n\ar[r]&S^{(n)}&\ar[l] S^{[n]},}$$
cf.\ \cite[\S\! 3.4]{HuyPI}.\smallskip

The BKR-equivalence (\ref{eqn:BKR}) is given by the equivariant Fourier--Mukai kernel
$\ko_{\kz_n}$, where $\kz_n\subset S^{[n]}\times S^n$ is Haiman's isospectral Hilbert scheme, i.e.\ $\kz_n$ is the reduction of the fibre product of $S^{[n]}\times_{S^{(n)}}S^n$,
cf.\ \cite[Thm.\ 6]{Haiman} or \cite[Cor.\ 5.5, preprint]{Krug}.\footnote{For $n=2$, the case of interest to us, the second projection of the subscheme $\kz_2\subset S^{[2]}\times (S\times S)$ is simply the blow-up $\text{Bl}_\Delta(S\times S)\to S\times S$ along the diagonal  and the first projection $\kz_2\cong\text{Bl}_\Delta(S\times S)\to S^{[2]}$  is the quotient by the ${\mathfrak S}_2$-action with ramification along the exceptional divisor.} In particular, the restriction of $(\alpha^{[n]})^{-1}\times \alpha^{\times n}\in \Br(S^{[n]}\times S^n)$ to $\kz_n$ is trivial and we can use $\ko_{\kz_n}$ 
as an equivariant
Fourier--Mukai kernel to get an exact functor
$$\Phi\colon \Db(S^{[n]},\alpha^{[n]})\to\Db(S^n,\alpha^{\times n})_{{\mathfrak S}_n}.$$
The composition $\Phi_L\circ\Phi$ of $\Phi$ with its left adjoint functor is a Fourier--Mukai endo-functor
of $\Db(S^{[n]},\alpha^{[n]})$ with a certain kernel $\kq$ and adjunction defines
a natural map $\varphi\colon\kq\to\ko_{\Delta}$ in the category $\Db(S^{[n]}\times S^{[n]},(\alpha^{[n]})^{-1}\times\alpha^{[n]})$. Locally over $S^{(n)}$, the map $\varphi$ is the adjunction morphism in the untwisted BKR-equivalence and thus an isomorphism. Therefore, $\Phi$ is fully faithful, cf.\ \cite[Cor.\ 1.23]{HuyFM}. Since $\Db(S^n,\alpha^{\times n})_{{\mathfrak S}_n}$
is a connected Calabi--Yau category, cf.\ \cite[\S\! 5]{BeckOber}, this suffices to deduce that $\Phi$ is actually an equivalence.
\end{proof}

%%%%%%%%%%%%%%%%%

\subsection{The Mukai lattice of the Hilbert square}\label{sec:MukaiLattice}
Recall from the introduction the definition of the Mukai lattice
$$\widetilde H(\ka_X^{[2]},\ZZ)\coloneqq \widetilde H(\ka_X,\ZZ)\oplus^\perp\ZZ\cdot\delta,$$
 of the Hilbert square $\ka_X^{[2]}$,\TBC{Could we also motivate this by arguing
 $\ka_X^{[2]}\subset\Db(X^{[2]})$ and then use the cohomology of $X^{[2]}$?}
which is naturally a Hodge structure of K3 type with $\delta$ of type $(1,1)$.
The definition was motivated by the fact that for $\ka_X\cong\Db(S)$, one
has $$\ka_X^{[2]}\cong\Db(S)^{[2]}\cong\Db(S^{[2]}),$$ see Proposition \ref{prop:twBKR},
and $$\widetilde H(S^{[2]},\ZZ)\cong \widetilde H(S,\ZZ)\oplus^\perp\ZZ\cdot \delta.$$
In fact, this determines $\widetilde H(\ka_X^{[2]},\ZZ)$ uniquely in the following sense.
The construction that associates with $X$ the Hodge structure $\widetilde H(\ka_X^{[2]},\ZZ)$ defines a morphism 
\begin{equation}\label{eqn:Period}
\kc\to D/{\rm O}(\Gamma),~X\mapsto \widetilde H(\ka_X^{[2]},\ZZ)
\end{equation} from the moduli space of smooth cubic fourfolds $X\subset\PP^5$ to the moduli space of Hodge structures of K3 type on the lattice $\Gamma\coloneqq\widetilde H({\rm K3},\ZZ)\oplus\ZZ(-2)$.
As a direct consequence of the density of the set of $X\in\kc$ with $\ka_X\cong \Db(S)$,
one then easily proves the next result.

\begin{lem} Assume $\kp\colon\kc\to D/{\rm O}(\Gamma)$ is a continuous
map such that for any $X\in\kc$ with $\ka_X\cong\Db(S)$ for some K3 surface $S$ the isomorphism type of the Hodge structure given by $\kp(X)$ is Hodge isometric to $\widetilde H(S^{[2]},\ZZ)$. Then $\kp$ coincides with {\rm(\ref{eqn:Period})}.\qed
\end{lem}

More evidence that the above is indeed the correct definition for the Mukai
lattice of $\ka_X^{[2]}$ would come from establishing that equivalences act on it. 
Recall that by \cite[Prop.\ 3.4]{HuyComp} any equivalence 
$\ka_X\cong\ka_{X'}$ that is of Fourier--Mukai type, i.e.\ such that the composition $\Db(X)\to\ka_X\cong\ka_{X'}\,\hookrightarrow \Db(X)$ with the natural inclusion and projection is a Fourier--Mukai functor, induces a Hodge isometry $\widetilde H(\ka_X,\ZZ)\cong\widetilde H(\ka_{X'},\ZZ)$. In fact, it was later proved by Li, Pertusi, and Zhang
that the hypothesis is automatic, see \cite[Thm.\ 1.3]{LPZ}. \smallskip

It seems plausible that similar techniques combined with the results of
Beckmann \cite{Beck} and Taelman \cite{Taelman}
could also prove that any
%\begin{prop}\label{prop:DbH}
equivalence $\ka_X^{[2]}\cong\ka_{X'}^{[2]}$ of Fourier--Mukai type 
between the Hilbert squares of the Kuznetsov components of two smooth cubic fourfolds naturally induces a Hodge isometry 
$$\widetilde H(\ka_X^{[2]},\ZZ)\cong\widetilde H(\ka_{X'}^{[2]},\ZZ).$$
Another approach to establish this would be by showing that any Fourier--Mukai equivalence
$\ka_X^{[2]}\cong\ka_{X'}^{[2]}$  deforms to an equivalence
$\ka_Y^{[2]}\cong\ka_{Y'}^{[2]}$ for which $\ka_Y\cong\Db(S)$ and
$\ka_{Y'}\cong\Db(S')$ and then \cite{Beck,Taelman} can be applied directly. 

%%%%%%%%%%%%%%%%%%%%%%
\section{Fano varieties possessing a Lagrangian fibration}
The aim of this section is to prove Theorem \ref{thm2} which we rephrase here as follows.

\begin{thm}\label{thm:rephrase}
Assume $d/2$ is a perfect square. Then for a Zariski dense
open subset of cubics $X$ in the Hassett divisor $\kc_d$ there exists an exact linear equivalence
$$\Db(F_X)\cong\ka_X^{[2]}.$$
\end{thm}

It should in principle be possible to extend the result to all $X\in \kc_d$ with $d$ as
above, but one would need more control over stability conditions on $F_X$ and $\ka_X^{[2]}$. %Ongoing work of Stellari et al on stability conditions on Hilbert schemes should be relevant here.
\smallskip

In addition to the arguments already used in the proof of Theorem \ref{thm1}, two
new ingredients will be needed: A birational correspondence between the Fano variety and a certain moduli spaces of twisted sheaves on curves in  a K3 surface and a twisted version of a result of Arinkin \cite{Arinkin} concerning derived categories of such moduli spaces recently established in \cite{Bottini}.

%%%%%%%%%%%%%%%%
\subsection{Twisted birational fourfolds}
We will need a twisted version of a result of Kawamata \cite{Ka} and Namikawa \cite{Na}. The result was originally proved by Addington, Donovan, and Meachan \cite[Prop.\ 1.3]{ADM}.\footnote{Thanks to N.\ Addington for the reference.} For the convenience of the reader we include a complete proof, also clarifying the assumptions.\smallskip

The Brauer group of a smooth projective variety is a birational invariant, cf.\ \cite[Cor.\ 6.2.11]{CTS}.
In particular, any birational correspondence $Z\sim Z'$ between two projective hyperk\"ahler manifolds induces naturally an isomorphism
\begin{equation}\label{eqn:Brbir}
\Br(Z)\cong \Br(Z').
\end{equation} between their Brauer groups.
 This isomorphism can also be described via the natural  Hodge isometry $H^2(Z,\ZZ)\cong H^2(Z',\ZZ)$, cf.\ \cite[Lem.\ 2.6]{HuyInv}. To see this, use the standard surjection $H^2(Z,\mu_\ell)\twoheadrightarrow \Br(Z)[\ell]$, which also holds for $Z\,\setminus\, P$, etc., and the standard comparison theorem for \'etale versus singular cohomology with finite coefficients.\smallskip

For Mukai flops one can be more specific, which will come in handy later. Recall that a birational correspondence between two (projective) hyperk\"ahler manifolds $Z\sim Z'$ is a Mukai
flop if it flops a Lagrangian $\PP^n\cong P\subset Z$ to its dual in $P^\ast\subset Z'$ and is an isomorphism $Z\,\setminus\, P\cong Z'\,\setminus\, P^\ast$ otherwise.  Thus, (\ref{eqn:Brbir}) in this case
reads \begin{equation}\label{eqn:BrBr}
\Br(Z)\cong\Br(Z\,\setminus\, P)\cong\Br(Z'\,\setminus\, P^\ast)\cong\Br(Z').
\end{equation}
The result of Kawamata and Namikawa then says that there exists an equivalence $\Db(Z)\cong \Db(Z')$, cf.\ \cite[Ch.\ 11]{HuyFM}.
To prove the twisted version of this statement, we will need the fact that there exist morphisms
\begin{equation*}\xymatrix{Z\ar[r]^-f& \bar Z&\ar[l]_{~f'}Z'}
\end{equation*} 
contracting $P$ and $P^\ast$ to the same singular point $z\in \bar Z$ and
otherwise inducing isomorphisms $Z\,\setminus\, P\cong\bar Z\,\setminus\,\{z\}\cong Z'\,\setminus\,P^\ast$. This allows us to compare Brauer classes on $Z$ and $Z'$ via $\bar Z$.

\begin{lem}\label{lem:BrBr}
For any $\alpha\in\Br(Z)$ corresponding to $\alpha'\in \Br(Z')$ under {\rm(\ref{eqn:BrBr})}
there exists a Brauer class $\bar\alpha\in\Br(\bar Z)$ with $\alpha=f^\ast\bar\alpha$ and $\alpha'=f'^\ast\bar\alpha$.
\end{lem}

\begin{proof} We follow Grothendieck's arguments to prove \cite[Thm.\ 7.1]{BrauerIII}.
Applying the Leray spectral sequence to $f$, we obtain the exact sequence
$$0\to\Pic(\bar Z)\to\Pic(Z)\to H^0(\bar Z, R^1f_\ast\GG_m)\to \Br(\bar Z)\to\Br(Z)\to H^1(\bar Z,R^1f_\ast\GG_m),$$
where we use $R^2f_\ast\GG_m=0$ to ensure $E_2^{0,2}=0$. Then observe $R^1f_\ast\GG_m=i_{z\ast}\Pic(P)$, which proves $H^0(\bar Z, R^1f_\ast\GG_m)=\Pic(P)$ and $H^1(\bar Z,R^1f_\ast\GG_m)=0$.
Thus, there exists an exact sequence
$$0\to\Pic(\bar Z)\to \Pic(Z)\to\Pic(P)\to \Br(\bar Z)\to \Br(Z)\to0$$
and similarly for the other projection $f'$.\footnote{In \cite{ADM} the authors make the simplifying assumption that $\Pic(Z)\to\Pic(P)$ is surjective, in which case $\Br(Z)\cong\Br(\bar Z)\cong \Br(Z')$.}
Using the commutativity of the diagram
$$\xymatrix@R=15pt{\Br(\bar Z)\ar[d]\ar[r]&\Br(Z)\ar[d]^-\cong\\
 \Br(\bar Z\,\setminus\,\{z\})\ar[r]^-\cong&\Br(Z\,\setminus\,P),}$$ one finds that
 the two subgroups  
 $\coker(\Pic(Z)\to\Pic(P))$ and $\coker(\Pic(Z')\to\Pic(P^\ast))$ of $\Br(\bar Z)$ are both
simply the kernel of the restriction map $\Br(\bar Z)\to\Br(\bar Z\,\setminus\,\{z\})$. This proves the existence of a (not necessarily unique) class $\bar\alpha\in \Br(\bar Z)$ as claimed.
\end{proof}

 We are now ready to prove the following twist of the result by Kawamata \cite{Ka} and Namikawa \cite{Na}, under additional assumptions already proved in \cite{ADM}.

\begin{prop}[Kawamata, Namikawa, Addington--Donovan--Meachan]\label{prop:KNADM}
Assume $Z\sim Z'$ is a Mukai flop between two projective hyperk\"ahler manifolds.
For any Brauer class $\alpha\in \Br(Z)$ and its associated Brauer
class $\alpha'\in \Br(Z')$ there exists an exact linear equivalence
$$\Db(Z,\alpha)\cong\Db(Z',\alpha').$$
\end{prop}

\begin{proof}
According to \cite{Ka,Na}, 
the Fourier--Mukai transform $\Phi_\kp\colon \Db(Z)\to\Db(Z')$ with
kernel the structure sheaf $\kp\coloneqq\ko_\Gamma$ of the subscheme $\Gamma\coloneqq \tilde Z\cup (P\times P^\ast)$ defines an equivalence, see also \cite[Ch.\ 11]{HuyFM}. Here, $\text{Bl}_P(Z)\cong\tilde Z\cong\text{Bl}_{P^*}(Z')$ is by construction of the Mukai flop a subscheme of $Z\times Z'$. The assertion that $\Phi_\kp$ is an equivalence is then the statement that the convolutions with $\kp_R=\kp_L\coloneqq \kp^\vee[2n]$
satisfy, cf.\ \cite[Ch.\ 5]{HuyFM}: 
\begin{equation}\label{eqn:convol}
\kp\ast\kp_L\cong\ko_{\Delta_{Z'}}\text{ and }\kp_R\ast\kp\cong\ko_{\Delta_{Z}}.
\end{equation}

We now apply Lemma \ref{lem:BrBr} and present the given Brauer class $\alpha\in \Br(Z)$ and the corresponding $\alpha'\in \Br(Z')$ as the pull-backs of a class $\bar\alpha\in\Br(\bar Z)$. Then
$\ko_\Gamma$ is a sheaf that can be considered as naturally twisted with respect to $\alpha^{-1}\times\alpha'\in \Br(Z\times Z')$. 
This allows one to use the same kernel $\kp=\ko_\Gamma$ to define the Fourier--Mukai transform
$\Phi_\kp\colon \Db(Z,\alpha)\to\Db(Z',\alpha')$. As (\ref{eqn:convol}) continues to hold, this defines again an equivalence.
\end{proof}

We will encounter two flops in the course of the proof of Theorem \ref{thm2}, both involving a polarised
K3 surface $(S,H)$ of degree two and Picard number one (or, slightly weaker, with all curves in $|H|$ integral):

\begin{enumerate}
\item Between the moduli space $M=M_S(0,H,1)$ of stable sheaves of rank
one and degree two on the fibres of the universal family $\kc\to |H|$ and the Hilbert scheme
$S^{[2]}$, see \S\! \ref{sec:obstruction}.
\item Between  the moduli space
$M_\alpha=M_{S,\alpha}(0,H,\chi)$ of $\alpha$-twisted sheaves of rank one on the
fibres of $\kc\to|H|$ and the Fano variety $F_X$, see \S\! \ref{sec:tw}.
\end{enumerate}
%%%%%%%%%%%%%%%%%%%%%%%%%%%
\subsection{Derived equivalence of twisted Jacobians}\label{sec:DbJac}
Assume $(S,H)$ is a polarised K3 surface and assume furthermore that all curves
$ C\in |H|$ are integral. Note that this is a Zariski open condition in the moduli
space of all $(S,H)$.
The relative Jacobian $$\Pic^0\coloneqq\Pic^0(\kc/|H|_{\text{sm}})\to |H|_{\text{sm}}$$
of the smooth part of the universal family $\pi\colon\kc\to |H|$ is naturally compactified
by the moduli space $M\coloneqq M_S(0,H,1-g)$ of $H$-stable sheaves of Mukai vector
$(0,H,1-g)$. Here, $g=(1/2)(H)^2+1$ is the genus of any curve $C\in |H|_{\text{sm}}$ and the generic sheaf parametrised by $M$ is indeed a line bundle of degree zero on such a curve:
\begin{equation}\label{eqn:PicM}
\xymatrix{\Pic^0\ar[d]\ar@{^(->}[r]&M\ar[d]\\
|H|_{\text{sm}}\ar@{^(->}[r]&|H|.}
\end{equation}
For generic $(S,H)$ the moduli space $M$ is a projective hyperk\"ahler manifold and the projection
given by mapping a sheaf to its support defines the Lagrangian fibration over $|H|$.\smallskip

Next we consider the twisted version. For this, we fix a Brauer class $\alpha\in\Br(S)$ and let
$M_{S,\alpha}(0,H)$ be the moduli space of $H$-semistable $\alpha$-twisted 
sheaves that are direct images of $\alpha$-twisted sheaves on curves $C\in |H|$.
We fix one connected component $M_\alpha\subset M_{S,\alpha}(0,H)$
and assume that $M_\alpha$ is a fine moduli space, i.e.\ that there exists a universal $(\alpha\times1)$-twisted sheaf $\kq$ on $S\times M_\alpha$.\smallskip

 Again, for generic $(S,H)$, the moduli space $M_\alpha$ is a projective hyperk\"ahler manifold of $\text{K3}^{[n]}$-type, see \cite{Yos}.
 Analogously to (\ref{eqn:PicM}), one has
\begin{equation}\label{eqn:PicM}
\xymatrix{\Pic_\alpha^\chi\ar[d]\ar@{^(->}[r]&M_\alpha\ar[d]\\
|H|_{\text{sm}}\ar@{^(->}[r]&|H|.}
\end{equation}
Here, $\Pic_\alpha^\chi\to |H|_{\text{sm}}$ is a torsor for the abelian
scheme $\Pic^0\to |H|_{\text{sm}}$.\footnote{To simplify the discussion, we do not discuss how to actually fix the numerical invariant $\chi$ here, which would involve choosing a $B$-field for $\alpha$. For us, $\Pic^\chi_\alpha$ is simply one of the components of the whole
$\Pic_\alpha$, as $M_\alpha$ is a component of $M_{S,\alpha}(0,H)$.}
Using the terminology of the Tate--{\v{S}}afarevi{\v{c}} group for the polarised K3 surface $(S,H)$
introduced in \cite{HuMa}, the twisted relative Picard $\Pic_\alpha^\chi$
is an element in $\Sha(S,H)$.

\begin{remark}\label{rem:BrCA}
Alternatively, one can consider the image of $\alpha$
under the natural map $$\Br(S)\to\Br(\kc)\to H^1(|H|_{\text{sm}},R^1\pi_\ast\Gm),$$ where $\pi\colon \kc\to|H|_{\text{sm}}$ is the projection, which then defines a twist $\Pic_\alpha$ of the relative Picard scheme $\Pic(\kc/|H|_{\text{sm}})\to |H|_{\text{sm}}$. This is nothing but the moduli space $\Pic_\alpha$ of which $\Pic_\alpha^\chi$ is one component, cf.\ \cite[Prop.\ 3.5]{HuMa}.\smallskip

We can also view this construction purely from the perspective of $\Pic^0$ as follows, cf.\ \cite[\S\! 3.1]{HuyPI}. Denote by $\kc^{[2]}$ the relative Hilbert scheme of $\kc\to|H|_{\text{sm}}$. Then there exists a natural map
$\Br(\kc)\to\Br(\kc^{[2]})$, $\alpha\mapsto \alpha^{[2]}$ and,
by means of the birational Abel--Jacobi map $\kc^{[2]}\to\Pic^{2}$ and
the isomorphism $\Pic^2\cong\Pic^0$, a natural isomorphism
$\Br(\kc^{[2]})\cong\Br(\Pic^0)$. Composing both gives
\begin{equation}\label{eqn:BrcurvPic0}
\Br(S)\to\Br(\kc)\to\Br(\kc^{[2]})\to\Br(\Pic^0).
\end{equation}

The class $\alpha^{[2]}$ is trivial on the fibres and thus its image defines a class
in $H^1(|H|_{\text{sm}},R^1f_\ast\Gm)$, where $f$ is the projection $\Pic^0\to |H|_{\text{sm}}$. Thus, $\alpha^{[2]}$ defines a twist of $\Pic(\Pic^0)=\Pic$ which is again just $\Pic_\alpha$.
\end{remark}

%Let us point out that our assumption on $M_\alpha$ to be a fine moduli space
%in particular implies that there exists a universal $(\alpha\times1)$-twisted bundle $\kq$
%on $\kc\times_{|H|}\Pic_\alpha^\chi$. This will be used crucially later.\smallskip

\begin{remark}
Note, that $\Pic^0=\Pic^0(\kc/|H|_{\text{sm}})$ and its torsor $\Pic_\alpha^\chi=\Pic^\chi_\alpha(\kc/|H|_{\text {sm}})$ are both relative moduli spaces of sheaves on the fibres of $\kc\to|H|_{\text{sm}}$, but the
first is also a moduli space of sheaves on the second. More precisely, since $\Pic^0\to|H|_{\text{sm}}$
is a principally polarised abelian scheme, the relative Jacobian of $\Pic^\chi_\alpha\to|H|_{\text{sm}}$ is $\Pic^0\to|H|_{\text{sm}}$.
\end{remark}

However, as $\Pic_\alpha^\chi\to |H|_{\text{sm}}$
does not admit a section, unless $\alpha$ induces a trivial class in $H^1(|H|_{\text{sm}},R^1\pi_\ast\Gm)\cong H^1(|H|_{\text{sm}},R^1f_\ast\Gm) $, there is usually no universal family $$\kp\to \Pic^0\times_{|H|}\Pic^\chi_\alpha.$$
Nonetheless, $\kp$ always exists as a $(\theta\times1)$-twisted line bundle, where $\theta\in\Br(\Pic^0)$ is the obstruction class to the existence of an untwisted universal bundle. The usual Fourier--Mukai
formalism then induces an equivalence
\begin{equation}\label{eqn:DbPic}
\Db(\Pic^0,\theta)\cong \Db(\Pic_\alpha^\chi).
\end{equation}

According to Bhatt \cite[App.\ A.1]{KLB} and using Remark \ref{rem:BrCA}, the class $\theta$
is, up to a sign, nothing but the image of $\alpha$ under (\ref{eqn:BrcurvPic0}).\smallskip

The next result is a compactified version of (\ref{eqn:DbPic}).

\begin{prop}\label{prop:obst}
The obstruction class $\theta\in\Br(\Pic^0)$ extends to a Brauer class on the moduli
space $M$, again called $\theta\in \Br(M)$, and {\rm (\ref{eqn:DbPic})} extends to an equivalence
\begin{equation}\label{eqn:DbPic2}
\Db(M,\theta)\cong \Db(M_\alpha).
\end{equation}
\end{prop}

\begin{proof} The result goes back to Arinkin \cite[Thm.\ C \& \S\! 7]{Arinkin} who proved that for an integral curve with planar singularities the Poincar\'e bundle extends to a sheaf on the square of the compactification of the Jacobian which taken as a Fourier--Mukai kernel induces an auto-equivalence.  For relative compactifications of Picard schemes of any degree the result was proved by Addington, Donovan, and Meachan \cite{ADM}. The result claimed
here is \cite[Thm.\ 3.3 \& Rem.\ 3.4]{Bottini}.
\end{proof}
%%%%%%%%%%%%%%%%%%%%%%%
\subsection{Computing the obstruction}\label{sec:obstruction}
From now on we assume that $(S,H)$ is a generic polarised K3 surface of degree two.
In particular, we may assume that all $C\in |H|$ are integral and that  the linear system $|H|$ defines a finite morphism $\varphi\colon S\to\PP^2$ of degree two branched over a curve of degree six. The curves $C\in |H|$
are the pre-images of lines in $\PP^2$. In this situation, we wish to relate the moduli 
space $M=M(0,H,-1)$ and the Hilbert scheme
$S^{[2]}$.\smallskip

 If $\xi\in S^{[2]}$ is not a fibre of $\varphi\colon S\to\PP^2$, then there exists exactly one line $\ell_\xi\subset \PP^2$ containing its image and hence exactly one curve $C_\xi\in|H|$, the pre-image of $\ell_\xi$, containing $\xi$. Tensoring its ideal sheaf ${\mathcal I}_\xi\subset\ko_{C_\xi}$ with $H$ defines a point ${\mathcal I}_\xi\otimes H\in M$. The map
$S^{[2]}\to M$, $\xi\mapsto {\mathcal I}_\xi\otimes H$, is regular and an open immersion on the complement of the plane $\PP^2\subset S^{[2]}$, $t\mapsto \varphi^{-1}(t)$. It is resolved by a Mukai flop
which is geometrically described by
\begin{equation}
\xymatrix{M&\ar[l]\ar[r]\kc^{[2]}&S^{[2]},} \xymatrix@C=17pt{{\mathcal I}_\xi\otimes H&\ar@{|->}[l]\ar@{|->}[r](\xi\subset C)&(\xi\subset S).}
\end{equation}
Here, as before, $\kc^{[2]}$ is the relative Hilbert scheme of zero-cycles of length two in the fibres $C$
of the family $\kc\to|H|$. For example, if $\xi\in S^{[2]}$ is a fibre of $\varphi$ over $t\in \PP^2$, then $\xi\subset C$ for all $C\in |H|$ with $t\in\varphi(C)$ and those
form a line in $|H|$.\smallskip

This birational correspondence between $M$ and $S^{[2]}$, which induces
an isomorphism between their Brauer groups,  cf.\ (\ref{eqn:BrBr}),
 allows one to identify the obstruction class $\theta\in \Br(M)$ in Proposition \ref{prop:obst}.

\begin{prop}\label{prop:IDBr}
Under the natural isomorphism $\Br(M)\cong\Br(S^{[2]})$ induced by the
Mukai flop above, the obstruction class $\theta$ is, up to a sign, identified with $\alpha^{[2]}$.
\end{prop}

\begin{proof} Since the restriction to any non-empty open subset is injective,
the assertion can be verified on the pre-images of the open set $|H|_{\text{sm}}$. \smallskip

According to  Remark \ref{rem:BrCA} and Bhatt's result \cite[App.\ A.1]{KLB}, 
the class $\theta$, up to a sign, pulls back under $\kc^{[2]}\to\Pic^0\subset M$
 to the pull back of $\alpha^{[2]}$ under the projection $\kc^{[2]}\to S^{[2]}$.
\end{proof}

%%%%%%%%%%%%%%%%%
\subsection{Twisted degree two K3 surfaces}\label{sec:tw}
We recall the following notation from \S\ \ref{sec:DbJac}: If $(S,H)$ is a polarised K3 surface of degree two with a Brauer class $\alpha\in \Br(S)$, then we denote by $M_{S,\alpha}(0,H)$ 
the moduli spaces of $H$-semi-stable $\alpha$-twisted sheaves on $S$ 
that are direct images of $\alpha$-twisted sheaves of rank one on curves $C\in |H|$.
This moduli space has infinitely many components and the Fano variety $F_X$ singles out one of them.

\begin{prop}\label{prop:K3degtwo}
Assume $d/2$ a perfect square. Then for all smooth cubic fourfolds
$X$ contained in a certain Zariski dense open subset of $\kc_d$
there exists a polarised K3 surface
$(S,H)$ of degree two together with a Brauer class $\alpha\in \Br(S)$ such that
$$\ka_X\cong\Db(S,\alpha).$$
Furthermore, the equivalence induces a birational correspondence between the Fano variety $F_X$ and one fine component $M_\alpha$ of $M_{S,\alpha}(0,H)$:
$$F_X\sim M_\alpha.$$
\end{prop}

\begin{proof}

As the assertion only concerns generic cubics in $\kc_d$ and is in itself  an open condition, we may assume  that $F_X$ %has Picard number two and 
is non-special, i.e.\ $(H^{2,0}\oplus H^{0,2})(F_X,\QQ)=0$.\footnote{The condition to be non-special is expected to be superfluous.} 
Then, according to \cite[Thm.\ 1.2]{HuMa}, $F_X$ is birational to a moduli space $M_\alpha$ of stable $\alpha$-twisted sheaves on a K3 surface $S$ that are direct images of $\alpha$-twisted sheaves of rank one on curves $C\in |H|$
in a certain complete linear system.
 Now, for a generic cubic $X$ in $\kc_d$, the line bundle $H$  is ample
with $2(H)^2=\dim M_\alpha=\dim F_X=4$. Hence,  $(S,H)$ is a polarised K3 surface of degree two. The formula also shows that $M_\alpha$ is projective, for the Mukai vector of the considered sheaves is primitive and the polarisation with respect to which stability is taken can be assumed to be generic.\smallskip

The birational correspondence $F_X\sim M_\alpha$ induces a Hodge isometry
\begin{equation}\label{eqn:HodgeFM}
H^2(F_X,\ZZ)\cong H^2(M_\alpha,\ZZ)
\end{equation}
and the two sides can be described in terms
of $\widetilde H(\ka_X,\ZZ)$ and $\widetilde H(S,\alpha,\ZZ)$ as 
$$H^2(F_X,\ZZ)=\lambda_1^\perp\subset \widetilde H(\ka_X,\ZZ)\text{ and } H^2(M_\alpha,\ZZ)\cong v^\perp\subset \widetilde H(S,\alpha,\ZZ).$$
These are results of Addington \cite[Cor.\ 8]{Add}, cf.\ \cite[Cor.\ 6.5.1]{HuyCubics},
and Yoshioka \cite[Thm.\ 3.19]{Yos}. Here, $\lambda_1\in \widetilde H^{1,1}(\ka_X,\ZZ)$ is the image of one of the two standard generators of $A_2$ under the distinguished embedding $A_2\,\hookrightarrow \widetilde H^{1,1}(\ka_X,\ZZ)$, cf.\ \cite[Ch.\ 6.5.4]{HuyCubics}, and $v\in \widetilde H^{1,1}(S,\alpha,\ZZ)$ is the Mukai vector of the sheaves parametrised by $M_\alpha$.\smallskip

The Hodge isometry (\ref{eqn:HodgeFM}) extends to a Hodge isometry
$$H^2(F_X,\ZZ)\oplus^\perp \ZZ\cdot\lambda_1\cong H^2(M_\alpha,\ZZ)\oplus^\perp \ZZ\cdot v$$
by  sending $\lambda_1$ to $\pm v$ and further, choosing the sign according to the
map  between the discriminant groups induced by (\ref{eqn:HodgeFM}), to a Hodge isometry
of the index two over-lattices
\begin{equation}\label{eqn:HH}
\widetilde H(\ka_X,\ZZ)\cong \widetilde H(S,\alpha,\ZZ).
\end{equation}
According to \cite[Thm.\ 1.4]{HuyComp}, the latter then proves $\ka_X\cong \Db(S,\alpha)$ for a Zariski open and dense set of cubics $X\in \kc_d$.\smallskip

It remains to show that the moduli space $M_\alpha$ is indeed fine.
For this we use the image $\lambda_2\in \widetilde H^{1,1}(\ka_X,\ZZ)$ of the second generator of $A_2$ and its image $w\in\widetilde H^{1,1}(S,\alpha,\ZZ)$ under
(\ref{eqn:HH}). The latter satisfies
$(v,w)=1$ and, therefore, the standard criterion for the existence of a universal family applies,
cf.\ \cite[Ch.\ 4.6]{HuLe}.\end{proof}

With some effort one can prove that the assertion of the proposition does in fact hold for all smooth cubics in $\kc_d$ for which $d/2$ is a perfect square. For this one uses  \cite[Cor.\ 29.7 \& Prop.\ 33.1]{BLMNPS} which extends \cite[Thm.\ 1.4]{HuyComp}
from the generic cubic in $\kc_d$ to all of them. Note that specialising the birational correspondence $F_X\sim M_\alpha$ again produces a birational correspondence.
However, in the discussion in \S\! \ref{sec:DbJac} and, in particular, in Proposition \ref{prop:obst} we do not know how to drop the assumption that $X\in \kc_d$ or the
corresponding surface $(S,H)$ are generic, and so Theorem \ref{thm:rephrase} is at this
time only proved for generic cubics $X$.

%%%%%%%%%%%%%%

%%%%%%%%%%%%%%
\subsection{End of proof of Theorem \ref{thm:rephrase}}
Assume $d/2$ is a perfect square. Then, according to Proposition
\ref{prop:K3degtwo}, there exists a non-empty Zariski open subset $U\subset \kc_d$
such that for any $X\in U\subset \kc_d$ one finds a polarised K3 surface $(S,H)$ of degree two with $\ka_X\cong\Db(S,\alpha)$
for a certain Brauer class $\alpha\in \Br(S)$. After possibly shrinking $U$, we may
assume that $H$ is very ample and that all curves in the linear system $|H|$ are integral.\smallskip

Furthermore, still according to Proposition \ref{prop:K3degtwo}, the Fano variety $F_X$ of $X$
is birational to a certain fine moduli space $M_\alpha$ of $H$-semi-stable $\alpha$-twisted sheaves, which are of rank one on the curves in $|H|$. Then, by the original untwisted result of
Kawamata and Namikawa, see Proposition \ref{prop:KNADM},
we know
$$\Db(F_X)\cong \Db(M_\alpha).$$
Proposition \ref{prop:obst} describes the right hand side as
$$\Db(M_\alpha)\cong \Db(M,\theta),$$
where the moduli space $M=M(0,H,-1)$ is a smooth compactification of
the relative Jacobian of the smooth family $\kc\to|H|_{\text{sm}}$ and $\theta\in\Br(M)$ is the obstruction to the existence of a twisted universal sheaf on $M\times_{|H|}M_\alpha$.\smallskip

By virtue of the discussion in \S\! \ref{sec:obstruction}, the moduli space $M$ is birational to the Hilbert scheme $S^{[2]}$ and under this birational correspondence $\theta\in \Br(M)$
corresponds to $\alpha^{[2]}\in \Br(S^{[2]})$, cf.\ Proposition \ref{prop:IDBr}.
Thus, by again applying Proposition \ref{prop:KNADM}, we find
$$\Db(M,\theta)\cong\Db(S^{[2]},\alpha^{[2]}).$$
Eventually, by the twisted BKR equivalence, cf.\ Proposition \ref{prop:twBKR}
$$\Db(S^{[2]},\alpha^{[2]})\cong\Db(S,\alpha)^{[2]}\cong\ka_X^{[2]}.$$
The string of equivalences combined now reads
$$\Db(F_X)\cong \Db(M_\alpha)\cong \Db(M,\theta)\cong\Db(S^{[2]},\alpha^{[2]})\cong
\Db(S,\alpha)^{[2]}\cong\ka_X^{[2]},$$
which concludes the proof of Theorem \ref{thm:rephrase}.\qed

%%%%%%%%%%%%%%%%%%%%%%%%%%
\section{Mukai lattices}
The main goal of this section is to prove Theorem \ref{thm3}.
 We will furthermore compare
the naive Mukai lattice $\widetilde H(F_X,\ZZ)$ with the integral Hodge structure
introduced by Taelman \cite{Taelman} and further developed by Beckmann \cite{Beck}.
The section concludes with a discussion of the auto-equivalence $\sign$ of $\ka_X^{[2]}$,
see \S\! \ref{sec:Hilbsquare}, and its action on $\widetilde H(\ka^{[2]},\ZZ)$.

%%%%%%%%%%%%%%%%%%%%

%%%%%%%%%%%%%%%%%%%%%
\subsection{Proof of Theorem \ref{thm3}} By declaring $U$ to be of type $(1,1)$,
we extend the weight two Hodge structure
$\widetilde H(\ka_X,\ZZ)$ to $\widetilde H(\ka_X,\ZZ)\oplus U$. We then consider the following two classes: $\lambda_1\in A_2\subset  \widetilde H^{1,1}(\ka_X,\ZZ)$
and $\gamma=e+f \in U\subset\widetilde H^{1,1}(\ka_X^{[2]},\ZZ)$, where
$e$ and $f$ are the standard generators of the hyperbolic plane $U$. Note that
$(\lambda_1)^2=(\gamma)^2=2$.\smallskip

Then there exist natural Hodge isometries
\begin{align*}
        \widetilde H(F_X,\ZZ) & = H^2(F_X,\ZZ) \oplus U \cong\lambda_1^\perp\subset \widetilde H(\ka_X,\ZZ) \oplus U \\
        \widetilde H(\ka_X^{[2]},\ZZ) &= \widetilde H(\ka_X,\ZZ) \oplus \ZZ\cdot\delta\cong\gamma^\perp \subset \widetilde H(\ka_X,\ZZ) \oplus U,
    \end{align*}
where the second isometry sends the class $\delta$ to $e-f\in U$.\smallskip

Next we observe that the two transcendental lattices $ T(F_X)\subset \widetilde H(F_X,\ZZ)$
    and $T(\ka_X^{[2]})\subset \widetilde H(\ka^{[2]}_X,\ZZ)$ are Hodge isometric to each other. Indeed, by construction we have
    $$T(F_X)\cong T(\ka_X)\cong T(\ka_X^{[2]}).$$
Furthermore, any Hodge isometry between the two can be extended to a Hodge isometry
 $$T(F_X)\oplus\ZZ\cdot\lambda_1\cong T(\ka_X^{[2]}) \oplus \ZZ\cdot \gamma$$ sending
 $\lambda_1$ to $\gamma$ and further to
\begin{equation}\label{eqn:HHka}
\widetilde H(\ka_X,\ZZ)\oplus U\cong \widetilde H(\ka_X,\ZZ)\oplus U.
\end{equation}
For the existence of the latter we are using a result of Nikulin, cf.\ \cite[Thm.\ 14.1.12]{HuyK3}, and
the fact that the complement $\widetilde H^{1,1}(F_X,\ZZ)$ of the sub-lattice $T(F_X)\oplus\ZZ\cdot\lambda_1\subset\widetilde H(\ka_X,\ZZ)\oplus U$ contains the hyperbolic plane $U$. Eventually, observe that since (\ref{eqn:HHka}) maps $\lambda_1$ to $\gamma$, it induces
a Hodge isometry between their orthogonal complements:
$$\widetilde H(F_X,\ZZ)\cong\lambda_1^\perp\cong\gamma^\perp\cong \widetilde H(\ka_X^{[2]},\ZZ)$$
which concludes the proof of Theorem \ref{thm3}.\qed
    
\begin{remark}
A priori, we have no control over the induced isometry between the algebraic parts
$$\widetilde H^{1,1}(F,\ZZ)\cong \widetilde H^{1,1}(\ka_X^{[2]},\ZZ).$$
However, for general $X\in \kc$, both lattices are of rank three. More precisely,
$$\widetilde H^{1,1}(F,\ZZ)\cong \ZZ \cdot g\oplus U~
\text{ and }  ~\widetilde H^{1,1}(\ka_X^{[2]},\ZZ)\cong\widetilde H^{1,1}(\ka_X,\ZZ)\oplus \ZZ\cdot \delta\cong A_2\oplus \ZZ\cdot \delta,$$ where $g$ is the Pl\"ucker polarisation satisfying $(g)^2=6$. Abstractly, the two lattices $\ZZ\cdot g\oplus U$ and $A_2\oplus\ZZ\cdot\delta$ are isometric, but it seems there is no distinguished isometry between the two.\footnote{For a reader of a more sceptical disposition this
may cast some doubt upon Conjecture \ref{conj}.}
At least, we could not come up with an isometry that would look more natural than others. Here
are two examples: (i) $e\mapsto -\lambda_2+\delta$, $f\mapsto \lambda_1+\lambda_2-\delta$, $g\mapsto2\lambda_1+4\lambda_2-3\delta$ and (ii)  $e\mapsto -\lambda_2-\delta$, $f\mapsto 3\lambda_1+8\lambda_2+7\delta$, $g\mapsto2\lambda_1+10\lambda_2+9\delta$.\smallskip

The situation is easier for the transcendental part, which was exploited already in the above proof.
For $X$ general, the transcendental lattices for both Hodge structures are
given by
$$T(F_X)\cong H^2(F_X,\ZZ)_{\text{pr}}\cong H^4(X,\ZZ)_{\text{pr}}^-~\text{ 
and }~T(\ka_X) \cong A_2^\perp \cong H^4(X,\ZZ)_{\text{pr}}^-,$$
both isometric to the abstract lattice $E_8(-1)^{\oplus 2}\oplus U^{\oplus 2}\oplus A_2(-1)$, cf.\ \cite[Ch.\ 6.5]{HuyCubics}.

\end{remark}

\begin{remark}
As a side remark, for a generic Pfaffian cubic fourfold Beauville and Donagi \cite{BD} constructed an isomorphism $F_X\cong S^{[2]}$,
which induces first the Pfaffian Hodge isometry 
$H^2(F_X,\ZZ)\cong H^2(S^{[2]},\ZZ)$ and then $\widetilde H(F_X,\ZZ)\cong \widetilde H(S^{[2]},\ZZ)$. The latter sends the Pl\"ucker polarisation $g$ on $F_X$ to $2h-5\delta\in H^2(S,\ZZ)\oplus\ZZ\cdot\delta\subset\widetilde H(S^{[2]},\ZZ)$. Observe, however, that
for the general smooth cubic $X$ there is no isometry $\ZZ\cdot g\oplus U\cong A_2\oplus \ZZ\cdot\delta$ that would send $g$ to a class $\alpha-5\delta$ with $\alpha\in A_2$.
\smallskip

In other words, the Pfaffian Hodge isometry does not deform to a Hodge isometry for the general cubic $X$ and, indeed, otherwise $F_X$ and $S^{[2]}$ would be birational even for
small deformations away from the Pfaffian locus. Similarly, one can show that
the equivalence $\Db(F_X)\cong \ka_X^{[2]}$ provided by Theorem \ref{thm2}  for Lagrangian fibred $F_X$ does not deform to the general $X$. Once the Beckmann--Taelman techniques are extended to also cover equivalences with $\ka_X^{[2]}$ this would also follow
from the description of the induced Hodge isometry between their associated Hodge structures. Indeed, the sign involution on $\ka_X^{[2]}$ defines an auto-equivalence of $\Db(F_X)$ which does not deform away from the Lagrangian locus.
\end{remark}
%%%%%%%%%%%%%%%%%%%%
\subsection{Comparison to the Beckmann--Taelman twist}\label{sec:BTtwist}
For a general hyperk\"ahler manifold $Z$, even when assumed to be of $\text{K3}^{[n]}$-type, the naive
weight two Hodge structure $\widetilde H(Z,\ZZ)=H^2(Z,\ZZ)\oplus U$ is not necessarily a derived
invariant. However, as shown by Taelman \cite{Taelman} for $n=2$ and later by Beckman 
\cite{Beck} for arbitrary $n$, a certain modification of it is. \smallskip

We briefly recall this construction for $n=2$, which makes use (but is ultimately independent) of an additional class $\delta\in H^2(Z,\ZZ)$ with $(\delta)^2=-2$ and
$\text{div}(\delta)=2$. Instead of $\widetilde H(Z,\ZZ)$ one considers its image $$\Lambda_Z\coloneqq B_{-\delta/2}(\widetilde H(Z,\ZZ))\subset \widetilde H(Z,\QQ)\coloneqq \widetilde H(Z,\ZZ)\otimes\QQ$$
under the B-field transform
$$\xymatrix@C=17pt{B_{-\delta/2}\colon \widetilde H(Z,\ZZ)\ar@{^(->}[r]& \widetilde H(Z,\QQ)%,~ \lambda  +r e + sf\mapsto   \lambda - r\frac{\delta}{2} + r e + \left(s  -\frac{(\delta,\lambda)}{2} -\frac{1}{4} \right)f. 
}$$
equipped with the weight two Hodge structure induced by the one on $\widetilde H(Z,\QQ)$.
Recall that the B-field transform $B_{-\delta/2}$ sends $\lambda\in H^2(Z,\ZZ)$ to $\lambda-(\delta.\lambda)/2\cdot f\in H^2(Z,\ZZ)\oplus U$, $e\in U$ to $-\delta/2+e-(1/4)\cdot f$, and preserves $f$. For the Mukai lattice of a K3 surface, $B_\gamma$
is nothing but multiplication with $\exp(\gamma)$, which might help to memorise these formulae.\smallskip

So, as a lattice, $\Lambda_Z$ is isometric to $\widetilde H(Z,\ZZ)$ via $B_{-\delta/2}$. However, unless $\delta$ is of type $(1,1)$, the isometry $B_{-\delta/2}$ does not respect the Hodge structure. Note that for $Z=S^{[2]}$ one can choose $\delta$ such that $2\delta$ is the class of the exceptional divisor of the Hilbert--Chow morphism. Thus, in this case the naive Hodge structure %$\widetilde H(S^{[2]},\ZZ)$
 is Hodge isometric to the
Beckmann--Taelman twist:% $\Lambda_{S^{[2]}}$:
$$\widetilde H(S^{[2]},\ZZ)\cong \Lambda_{S^{[2]}}.$$
This can be generalised easily to the following.

\begin{lem}
Assume the class $-\delta/2$ is of the form $\gamma_1+\gamma_2$
with $\gamma_1\in H^{1,1}(Z,\QQ)$ and $\gamma_2\in H^2(Z,\ZZ)$. Then there
exists a Hodge isometry between  the  naive Hodge structure and the
Beckmann--Taelman twist:
\begin{equation}\label{eqn:naiveHilb}
\widetilde H(Z,\ZZ)\cong \Lambda_Z.
\end{equation}
\end{lem}

\begin{proof} Under the given assumptions, the B-field transform
$B_{-\delta/2}$ is the composition of the two B-field transforms associated with $\gamma_1$ and $\gamma_2$, i.e.\ $
        B_{-\delta/2} = B_{\gamma_1} \circ B_{\gamma_2}$.\smallskip
        
        Since $\gamma_2$ is integral, the B-field transform $B_{\gamma_2}$ is invertible and so $B_{\gamma_2}(\widetilde H(Z,\ZZ))=\widetilde H(Z,\ZZ)$ as 
    sublattices of $\widetilde{H}(Z,\QQ)$. Furthermore, since $\gamma_1$ is algebraic,
    $
    B_{\gamma_1}\colon B_{\gamma_2}(\widetilde{H}(Z,\ZZ)) \congpf \Lambda_{Z}
    $ is a Hodge isometry. Combining both observations proves the assertion.
\end{proof}

The lemma allows us to prove a version of (\ref{eqn:naiveHilb})  for Fano varieties of lines.

\begin{prop}\label{prop:BT=naive}
For the Fano variety $F_X$ of lines contained in a smooth cubic fourfold, there exists a Hodge isometry
between  the  naive Hodge structure and the Beckmann--Taelman twist: $$\widetilde H(F_X,\ZZ)\cong \Lambda_{F_X}.$$
\end{prop}

\begin{proof} Denote by $g\in H^2(F_X,\ZZ)$ the class of the Pl\"ucker polarisation. 
If $-\delta/2$ can be written as $-\delta/2=p\cdot g+ \gamma$ with
$p\in \QQ$ and $\gamma\in H^2(F_X,\ZZ)$, then the lemma applies and proves the claim.\smallskip

There are various ways to achieve this. For example, according to Beauville and Donagi \cite{BD}, when specialising $X$ to a Pfaffian cubic,
the Fano variety of lines $F_X$ specialises to the Hilbert scheme $S^{[2]}$ of a polarised K3 surface $(S,h)$ of degree $14$. Furthermore, under the induced Hodge isometry
the Pl\"ucker polarisation
$g$ becomes $2h-5\delta$, cf.\ \cite[Ch.\ 6.2.4]{HuyCubics}. In other words,
 $-\delta/2=g/2+2\delta-h$, which still holds after parallel transport back to the original $F_X$.
\end{proof}
 %Alternatively, on can specialise the cubic $X$ to the determinantal one 
%and again the variety of lines $F_X$ specialises to the Hilbert square $S^{[2]}$ of a polarized K3 surface $(S,h)$, this time of degree $2$. According to van den Dries \cite[Prop.\  3.7.1]{vandenDries}, the Pl\"ucker polarization $g$ corresponds in this case to $2h - \delta$
%on $S^{[2]}$, i.e.\ $-\delta/2=g/2-h$.

%%%%%%%%%%%%%%%%%%%%%%%%%%
\subsection{The involution of the Hilbert square}\label{sec:sigma}

We consider the involution $\sigma\colon\ka_X^{[2]}\to \ka_X^{[2]}$ induced by the
definition of $\ka_X^{[2]}$ as the equivariant category $$\ka_X^{[2]}\cong(\ka_X\boxtimes\ka_X)_{{\mathfrak S}_2},$$ see \S\! \ref{sec:Hilbsquare} and \cite{Elagin}.\smallskip

If $\ka_X\cong\Db(S)$, then by BKR
 $\ka_X^{[2]}\cong\Db(S)^{[2]}\cong\Db(S^{[2]})$ and so, according to \cite{Taelman},
 the induced auto-equivalence $\sigma\colon \Db(S^{[2]})\to\Db(S^{[2]})$ naturally acts on $\widetilde H(S^{[2]},\ZZ)$,
 where we use Proposition \ref{prop:BT=naive}.

\begin{lem}
The involution $\sigma\colon \Db(S^{[2]})\to\Db(S^{[2]})$ acts 
as $-{\rm id}_{\widetilde H}\oplus {\rm id}_\delta$ on
 $$\widetilde H(\Db(S)^{[2]},\ZZ)=\widetilde H(S^{[2]},\ZZ)=\widetilde H(S,\ZZ)\oplus\ZZ\cdot\delta.$$\end{lem}

\begin{proof} This result is due to Beckmann \cite[Prop.\ 7.1]{Beck}.
\end{proof}

According to \S\! \ref{sec:MukaiLattice}, the involution $\sigma\colon \ka_X^{[2]}\to\ka_X^{[2]}$ is expected to act on $\widetilde H(\ka_X^{[2]},\ZZ)$ . In fact, as it clearly deforms to a situation in which $\ka_X\cong\Db(S)$, this follows from the work of
Beckman \cite{Beck} and Taelman \cite{Taelman}. Here, we are using that the sign
involution of $\ka_X^{[2]}$ or $\Db(S^{[2]})$ is preserved by auto-equivalences
coming from auto-equivalences of $\ka_X$ resp.\ $\Db(S)$. By continuity the lemma implies the following.

\begin{cor}
The involution $\sigma$  acts as $-{\rm id}_{\widetilde H}\oplus {\rm id}_\delta$
on $\widetilde H(\ka_X^{[2]},\ZZ)=\widetilde H(\ka_X,\ZZ)\oplus\ZZ\cdot\delta$.\qed
\end{cor}

\TBC{Can we compare this with the action for $X^{[2]}$ and then restrict?}

%%%%%%%%%%%%%%%%%%%%%

%%%%%%%%%%%%%%%%%%%%%%%

\end{document}